\documentclass[11pt, reqno]{amsart}
\usepackage[utf8]{inputenc}
\usepackage{graphicx} 
\usepackage[margin=1in]{geometry}

\usepackage{amsmath,amssymb,amsthm,bbm}
\usepackage{mathtools}
\usepackage{mathabx}
\usepackage{url}
\usepackage{color}
\usepackage{tikz}
\usepackage{soul}
\usepackage{comment}

\theoremstyle{definition} 
\theoremstyle{definition} 
\theoremstyle{definition} \newtheorem*{axi*}{Axiom}
\theoremstyle{definition} \newtheorem{prp}{Proposition}[section]
\theoremstyle{definition} 
\theoremstyle{definition} \newtheorem{lem}[prp]{Lemma}
\theoremstyle{definition} \newtheorem{thm}[prp]{Theorem}
\theoremstyle{definition} \newtheorem{cor}{Corollary}[section]
\theoremstyle{definition} 
\theoremstyle{remark} \newtheorem*{cmt}{Remark}
\newtheorem*{thm*}{Theorem}

\theoremstyle{remark} \newtheorem*{rmk}{Remark}

\newcommand \Z {\mathbb Z}

\title[Fuglede's Conjecture on Cyclic Groups]{Fuglede's Conjecture on Cyclic Groups of Square-Free Order: \\The Case of Rapidly Growing Prime Factors}

\author{Gábor Somlai}
\address{School of Mathematics and Statistics, The University of Melbourne, Parkville, VIC 3010, Australia, on unpaid leave at Eötvös Loránd University}
\email{gabor.somlai@unimelb.edu.au}
\email{gabor.somlai@ttk.elte.hu}
\thanks{Research supported by supported by   ARC Discovery Project DP250104965 and OTKA STARTING Grant 150576}
\begin{document}

\begin{abstract}
    The main result of the paper is that an inductive argument is established to prove the Fuglede's conjecture for an infinite sequence of square-free order cyclic groups. The tile-to-spectral direction of the conjecture holds for all square-free order cyclic groups, whereas we prove the spectral-to-tiling direction only for those groups among them whose prime factors grow rapidly. To achieve this, we develop an inductive technique with roots in a previous paper by the author co-authored with Fallon, Kiss, and Mayeli. Following recent preprints, Fuglede's conjecture remains open only for finite cyclic groups, and until now, there was no infinite family of cyclic groups for which the conjecture was known that possessed an arbitrary number of distinct divisors. 
\end{abstract}
\maketitle

\section{Introduction}

Let $\Omega\subseteq\mathbb R^d$ be a bounded measurable set of positive
measure. We say that $\Omega$ is \emph{spectral} if there exists a set
$\Lambda\subseteq\mathbb R^d$ such that
\[
    \bigl\{e^{2\pi i\langle \lambda,x\rangle}:
           \lambda\in\Lambda\bigr\}
\]
is an orthogonal basis of the Hilbert space $L^2(\Omega)$. In this case, $\Lambda$ is called
a \emph{spectrum} of $\Omega$. We say that $\Omega$ tiles $\mathbb R^d$ by
translations if there is a set $T\subseteq\mathbb R^d$ such that the
translates $\Omega+t$, $t\in T$, form a partition of $\mathbb R^d$ into pairwise disjoint sets up to
sets of measure zero. Fuglede conjectured that these two properties are
equivalent \cite{Fuglede1974}.

The conjecture is false in all dimensions $d\geq 2$. Terence Tao first
disproved the spectral-to-tile implication in dimensions $d\geq 5$
\cite{Tao2004}; Matolcsi subsequently obtained a four-dimensional
counterexample \cite{Matolcsi2005}, and Kolountzakis and Matolcsi reduced
the dimension to three \cite{KolountzakisMatolcsi2006Hadamard}. In the
opposite direction, translational tiles without spectra were constructed
in dimensions $d\geq 5$ by Kolountzakis and Matolcsi
\cite{KolountzakisMatolcsi2006Tiles}, in dimension four by Farkas and
R\'ev\'esz \cite{FarkasRevesz2006}, and in dimension three by Farkas,
Matolcsi, and M\'ora \cite{FarkasMatolcsiMora2006}. Very recently, Tao
Zhang constructed counterexamples to both implications in dimension two
\cite{Zhang2026DimensionTwo}. The one-dimensional case is therefore the
only remaining Euclidean case.

Finite cyclic groups are central to the one-dimensional problem.
Dutkay and Lai proved a series of reductions between Fuglede's conjecture
on $\mathbb R$, on $\mathbb Z$, and on finite cyclic groups
\cite{DutkayLai2014}. For the spectral-to-tile direction, the reverse
reduction depended on the rationality of one-dimensional spectra.
Fu and Song have recently uploaded a preprint to ArXiv which proves that every normalized spectrum of a
bounded spectral set on the real line is rational \cite{FuSong2026}.
Consequently,
\[
 \text{Fuglede's conjecture on }\mathbb R
 \quad\Longleftrightarrow\quad
 \text{Fuglede's conjecture on }\mathbb Z_N
 \text{ for every }N\in\mathbb N.
\]
Thus, after the failure of the conjecture in dimension two, the case of finite
cyclic groups remains the only general unsolved case.

We identify the dual group of $\mathbb Z_N$ with $\mathbb Z_N$ and write
$    \zeta_N=e^{2\pi i/N}$.
A set $A\subseteq\mathbb Z_N$ is a translational tile if there exists
$B\subseteq\mathbb Z_N$ such that every element of $\mathbb Z_N$ has a
unique representation as $a+b$, with $a\in A$ and $b\in B$. We then
write
\[
    A\oplus B=\mathbb Z_N.
\]
The non-empty set $A$ in $\Z_N$ is spectral if there is a set
$\Lambda\subseteq\mathbb Z_N$, with $|A|=|\Lambda|$, such that
\begin{equation}\label{eq:spectral}
    \sum_{a\in A}\zeta_N^{(\lambda-\lambda')a}=0
    \qquad
    \text{whenever }\lambda,\lambda'\in\Lambda,\quad
    \lambda\neq\lambda'.
\end{equation}
Notice that despite the fact that $a$ is an element of $\Z_n$, equation \eqref{eq:spectral} is still well-defined,
Using representatives for the elements of $\Z_n$ in $\{0,\ldots,N-1\}$, the \emph{mask polynomial}
of $A$ is
\[
    A(X)=\sum_{a\in A}X^a.
\]
Notice that the mask polynomial is an element of $\Z[x]/(x^n-1)$.

The tiling identity $A\oplus B=\mathbb Z_N$ is equivalent to
\[
    A(X)B(X)
    \equiv 1+X+\cdots+X^{N-1}
    \pmod{X^N-1}.
\]
Mask polynomials also provide a cyclotomic formulation of spectrality and this is the language we use when we prove our main result.
For $d\in\mathbb Z_N$, let
$    \operatorname{o}_N(d)=\frac{N}{\gcd(N,d)}$
denote the order of $d$ in the additive group $\Z_N$. Since $\zeta_N^d$ is a primitive
$\operatorname{o}_N(d)$th root of unity,
\[
    A(\zeta_N^d)=0
    \quad\Longleftrightarrow\quad
    \Phi_{\operatorname{o}_N(d)}(X)\mid A(X),
\]
where $\Phi_m$ denotes the $m$th cyclotomic polynomial. Hence
$\Lambda$ is a spectrum of $A$ precisely when $|A|=|\Lambda|$ and
\[
    \Phi_{\operatorname{o}_N(\lambda-\lambda')}(X)
        \mid A(X)
    \qquad
    \text{for all distinct }\lambda,\lambda'\in\Lambda.
\]
This relationship between Fourier zeros and cyclotomic divisibility is
one of the main tools in the study of Fuglede's conjecture in cyclic
groups \cite{CovenMeyerowitz1999,Laba2002,KissMalikiosisSomlaiVizer2020,Malikiosis2022}.

Substantial but still fairly few classes of cyclic groups are now known to satisfy Fuglede's
conjecture. Here and below, the displayed primes are distinct unless
otherwise indicated. The conjecture holds for prime powers (\L aba
\cite{Laba2002}), for groups of order $p^nq$
(Kolountzakis and Malikiosis \cite{MalikiosisKolountzakis2017}), and for groups of order $pqr$
by Shi \cite{Shi2019}. Somlai proved the spectral-to-tile direction for
$\mathbb Z_{p^2qr}$ \cite{Somlai2023}, while Kiss, Malikiosis, Somlai,
and Vizer proved the conjecture for $\mathbb Z_{pqrs}$; together with
the earlier cases, their result covers all cyclic groups whose order has
at most four prime factors counted with multiplicity
\cite{KissMalikiosisSomlaiVizer2022}. Zhang subsequently proved the
conjecture for the larger family $\mathbb Z_{p^nqr}$
\cite{Zhang2024GroupRing}. For two-prime orders, Malikiosis \cite{Malikiosis2022} proved the
conjecture for $\mathbb Z_{p^mq^n}$, with $p<q$, whenever
\[
    m\leq 9,\qquad\text{or}\qquad n\leq 6,
\]
and also whenever
\[
    p^{m-2}<q^4.
\]
 A preprint by the author and Fallon, Kiss and Mayeli develops a
large-prime approach to the spectral-to-tile direction in
$\mathbb Z_{p^2q^2r}$ under the assumption
$p^2q^2\leq r$ \cite{FallonKissMayeliSomlai2023}. Notice that similar approach appears in the work of \L aba and Londner \cite{LabaLondner2025Splitting} for the Coven--Meyerowitz conjecture. 

The tile-to-spectral direction is closely connected with the
Coven--Meyerowitz conjecture in integer tilings. For a finite set
$A\subseteq\mathbb Z$ with mask polynomial $A(X)$, define
\[
    \mathcal S_A
      =\bigl\{p^\alpha:
         \Phi_{p^\alpha}(X)\mid A(X)\bigr\}.
\]
The Coven--Meyerowitz conditions are
\[
 \tag{T1}
 |A|=\prod_{s\in\mathcal S_A}\Phi_s(1)
\]
and
\[
 \tag{T2}
 \Phi_{s_1\cdots s_k}(X)\mid A(X)
\]
whenever $s_1,\ldots,s_k\in\mathcal S_A$ are powers of distinct
primes. Coven and Meyerowitz proved that \textup{(T1)} and \textup{(T2)}
are sufficient for tiling, that \textup{(T1)} is necessary for every
finite tile, and that \textup{(T2)} is necessary when $|A|$ has at most
two distinct prime divisors \cite{CovenMeyerowitz1999}. Their
construction produces a standard tiling complement directly from the
prime-power cyclotomic data $\mathcal S_A$.

The assertion that \textup{(T2)} is necessary for every finite tile is
usually referred to as the Coven--Meyerowitz conjecture. It was
explicitly discussed and popularized in a 2011 blog post of Tao
\cite{Tao2011CM}. In the comments to that post, \L aba recorded an
argument, following observations of Meyerowitz, which resolves the
square-free case; see also the self-contained treatment in
\cite{Shi2019}. In particular, if $N$ is square-free, then every tile
$A\subseteq\mathbb Z_N$ is a complete set of representatives for the
cosets of a subgroup. Consequently, a tile of cardinality $k$ has the
standard subgroup complement
    $k\mathbb Z_N$. The main point of this conversation is that there are canonical tiling complements for certain cyclic groups. In the case of square-free cyclic groups the this standard tiling complement is just a subgroup that only depends on the size of the tile. 

More recently, \L aba and Londner used a different starting point, the so-called Sands's theorem and developed a combinatorial method for integer tilings
\cite{LabaLondner2022Methods}. They proved the Coven--Meyerowitz
conjecture for tilings of period (or for the cyclic group $\Z_M$)
\[
    M=(p_1p_2p_3)^2,
\]
first when the three primes are odd
\cite{LabaLondner2023Odd}, and subsequently without the oddness
assumption \cite{LabaLondner2025Even}. Their splitting method  \cite{LabaLondner2025Splitting} also
establishes the conjecture for several families in which one prime is
large relative to the remaining prime powers, including
\[
    M=p_1^{n_1}p_2^{n_2}p_3^{n_3},
    \qquad
    p_1>p_2^{\,n_2-1}p_3^{\,n_3-1},
\]
and
\[
    M=p_1^{n_1}p_2^2p_3^2p_4^2,
    \qquad
    p_1>p_2p_3p_4.
   \]
These results suggest that standard tiling complements are especially
effective when a dominant prime is present. Their approach is inductive and even more, the same holds for the case when $M$ is square-free and such an inductive proof has not been available for the spectral-to-tile direction. 

Motivated by this observation, and by the large-prime reduction proposed
in \cite{FallonKissMayeliSomlai2023}, we consider the following
stability question:
\[
    \mathrm{S\!-\!T}(\mathbb Z_n)
    \quad\stackrel{?}{\Longrightarrow}\quad
    \mathrm{S\!-\!T}(\mathbb Z_n\times\mathbb Z_p),
    \qquad p>n,
\]
where $\mathrm{S\!-\!T}(G)$ denotes the assertion that every spectral
subset of $G$ tiles $G$. Our main result gives an affirmative answer
when $n$ is square-free.
\begin{thm}\label{thm:induction}
Let $n$ be square-free and let $p>n$ be prime. Suppose that
\[
    \mathrm{S\!-\!T}(\mathbb Z_n)
\]
holds. Then
\[
    \mathrm{S\!-\!T}
       (\mathbb Z_n\times\mathbb Z_p)
\]
holds. 
\end{thm}
The fact that tile-to-spectral direction of Fuglede's conjecture hold was published by Shi \cite{Shi2019}, after it was pointed out by \L aba and Meyerowitz on a blogpost of Tao that the result follows from a theorem of Tijdeman, sometimes also called the dialation lemma.
As a corollary we obtain immediately the following. 
\begin{thm}
Let $M=\prod_{i=1}^k p_i$ be square-free ($p_i$ are different primes). Assume that $p_{j+1}>\prod_{i=1}^{j}p_i$ for $j=1,2,\ldots, k-1$. Then Fuglede's conjecture holds in $\Z_M$. 
\end{thm}

In the second section of the paper, we introduce the techniques from Fourier analysis used in studying Fuglede's conjecture, focusing on when one or several cyclotomic polynomials divide a polynomial with integer coefficients. The third section contains the proof of the main result. The paper relies on several previous results, drawing primarily on a joint paper by the author with Fallon, Kiss, and Mayeli, while also using a seemingly minor but all the more useful observation by \L aba and Marshall. The paper includes the proofs of these earlier results to ensure the presentation is self-contained.

\section{Cyclotomic divisibility and the cube rule}
\label{sec:cyclotomic-divisibility}

In this section we collect several standard consequences of cyclotomic
divisibility that will be used throughout the paper. We allow sets to
be replaced by multisets. Thus, if $A$ is a multiset in $\mathbb Z_M$,
we write $w_A(a)$ for the multiplicity of $a$ and define
\[
    m_A(X)=\sum_{a\in\mathbb Z_M}w_A(a)X^a,
    \qquad
    |A|=m_A(1)=\sum_{a\in\mathbb Z_M}w_A(a).
\]
For a set, of course, $w_A$ is its characteristic function.

If $N\mid M$, the projection of $A$ to $\mathbb Z_N$ is the multiset
$A\bmod N$ defined by
\[
    w_{A\bmod N}(x)
       =
    \sum_{\substack{a\in\mathbb Z_M\\a\equiv x\;(\mathrm{mod}\,N)}}
       w_A(a).
\]
Its mask polynomial satisfies $
    m_{A\bmod N}(X)
       \equiv m_A(X)
       \pmod{X^N-1}.
$
It follows that
\begin{equation}
\label{eq:cyclotomic-projection}
    \Phi_N(X)\mid m_A(X)
    \quad\Longleftrightarrow\quad
    \Phi_N(X)\mid m_{A\bmod N}(X).
\end{equation}
Thus, when studying divisibility by $\Phi_N$, we may always project the
multiset to $\mathbb Z_N$. The projection is a surjective homomorphism from $\Z_M$ to $\Z_N$ (e.g. $x \mapsto \frac{M}{N}x$).

\subsection{Cuboids and the cube rule}
Although we will not explicitly use the cuboid methods—that is, the cube rule—but only their consequences, we present the basic method \cite{KissMalikiosisSomlaiVizer2020} that is the foundation for the statements of this in section.

Let
$    N=\prod_{i=1}^r p_i^{\alpha_i}$,
where $p_1,\ldots,p_r$ are distinct primes. An \emph{$N$-cuboid} is a
signed multiset $\Delta$ in $\mathbb Z_N$ whose mask polynomial has the
form
\[
    m_\Delta(X)
      =
    X^c\prod_{i=1}^r
       \left(1-X^{u_iN/p_i}\right),
    \qquad
    (u_i,p_i)=1.
\]
The nonzero coefficients of $m_\Delta$ are $1$ and $-1$. For a multiset
$A$ in $\mathbb Z_M$, where $N\mid M$, define its evaluation on
$\Delta$ by
\[
    A^N[\Delta]
       =
    \sum_{x\in\mathbb Z_N}
       w_{A\bmod N}(x)w_\Delta(x).
\]
The following is the standard cube rule for cyclotomic
divisibility.
\begin{prp}[Cube rule]
\label{prop:cuboid-criterion}
Let $N\mid M$, and let $A$ be an integer-valued multiset in
$\mathbb Z_M$. The following are equivalent:
\begin{enumerate}
    \item $\Phi_N(X)\mid m_A(X)$;
    \item
    $
        A^N[\Delta]=0
    $
    for every $N$-cuboid $\Delta$;
    \item modulo $X^N-1$, the projection of $A$ is an integer linear
    combination of prime-order cosets (some people call them fibers):
    \[
        m_A(X)
        \equiv
        \sum_{\substack{q\mid N\\q\ {\rm prime}}}
             P_q(X)\Phi_q\left(X^{N/q}\right)
        \pmod{X^N-1},
    \]
    where $P_q(X)\in\mathbb Z[X]$.
\end{enumerate}
\end{prp}
The equivalence between cyclotomic divisibility and the coset or fiber
decomposition is the classical de Bruijn \cite{deBruijn1953}--R\'edei \cite{Redei1954}--Schoenberg \cite{Schoenberg1964} theorem on
vanishing sums of roots of unity. We would also highlight the result of Lam and Leung \cite{LamLeung2000} on vanishing sums of roots of unities, for quantitative bounds on the subsets of cyclic groups having at least one Fourier root.
For the equivalent cuboid formulation applied for Fuglede's conjecture, see
\cite[Section~3]{KissMalikiosisSomlaiVizer2020} and
\cite[Proposition~4.2]{LabaMarshall2022}.

When $N$ is square-free, the geometric meaning is particularly simple.
By the Chinese remainder theorem,
$
    \mathbb Z_N
       \cong
    \mathbb Z_{p_1}\times\cdots\times\mathbb Z_{p_r}$.
An $N$-cuboid can then be identified with the following Descartes product, a cube
\[
    C=\prod_{i=1}^r\{a_i,b_i\},
    \qquad a_i\neq b_i ~for ~i=1,\ldots, r ,
\]
having $2^r$ vertices. Fixing a vertex $c_0\in C$, the cuboid criterion
becomes
\begin{equation}
\label{eq:cube-rule}
    \sum_{c\in C}
       (-1)^{d_H(c,c_0)}w_A(c)=0,
\end{equation}
where $d_H$ denotes the Hamming distance. We say that $A$ satisfies the
\emph{$N$-cube rule} if \eqref{eq:cube-rule} holds for every such cube.
Thus, for square-free $N$,
$
    \Phi_N(X)\mid m_A(X)$
 if and only if 
   $ A$ satisfies the $N$-cube rule.

It is important that the fiber decomposition in
Proposition~\ref{prop:cuboid-criterion} is an integer linear
combination. In general, its coefficients need not be nonnegative, but the positivity of the coefficients holds if $N$ has at most two different prime divisors.

The following is a basic and well-known observation in the theory of cyclotomic divisibility. 
\begin{cor}
\label{cor:prime-divides-cardinality}
Let $p$ be prime divisor of $M$. Let $A$ be a subset of $\Z_M$. 
\begin{enumerate}
    \item 
If
$
    \Phi_p(X)\mid m_A(X)$,
then $   p\mid |A|$.
More generally, if $\Phi_{p^\alpha}(X)\mid m_A(X)$ for some
$\alpha\geq1$, then $p\mid |A|$.
\item If
$
    \Phi_p(X)\mid m_A(X)$, then each $\frac{M}{p}$ coset contains the same amount of elements of $A$.
\end{enumerate}
\end{cor}
\begin{proof}
\begin{enumerate}
    \item 
The prime-power statement also follows immediately by writing
\[
    m_A(X)=\Phi_{p^\alpha}(X)Q(X)
\]
and evaluating this equation at $X=1$, since
$    \Phi_{p^\alpha}(1)=p$ for $\alpha \ge 1$.
\item 
The statement follows from the fact that $\Phi_p=\sum_{i=0}^{p-1}x^i$.
\end{enumerate}
\end{proof}

\subsection{The cube rule on cosets}

We next record a levelwise version of the cube rule. Although the
statement will be useful below, it is a standard consequence of the
cuboid criterion.

Let $ 
    M=kp$,
where $p$ is prime and $(k,p)=1$. Using the Chinese remainder theorem,
we identify
\[
    \mathbb Z_{kp}\cong\mathbb Z_k\times\mathbb Z_p.
\]
For $j\in\mathbb Z_p$, define the $j$th \emph{$p$-level} of a multiset
$A\subseteq\mathbb Z_{kp}$ by
\[
    w_{A_j}(x):=w_A(x,j),
    \qquad x\in\mathbb Z_k.
\]
Such a set $A_j$ is the intersection of $A$ with a coset of $\Z_k$ so we may identify this coset with $\Z_k$ and we consider $A_j$ as a subset of $\Z_k$ as well. 
If $d\mid k$, divisibility of $m_{A_j}$ by $\Phi_d$ is understood after considering this identification of the cosets of $\Z_k$ and of course
projecting $A_j$ from $\mathbb Z_k$ to $\mathbb Z_d$.

The following Lemma is contained in \cite{LabaMarshall2022}.   
\begin{lem}[Levelwise cube rule, \L aba-Marshall]
\label{lem:levelwise-cube-rule}
Let $d\mid k$, and suppose that
\[
    \Phi_{dp}(X)\mid m_A(X).
\]
Then, for every $d$-cuboid $\Delta$, the quantities
\[
    A_j^d[\Delta],
    \qquad j\in\mathbb Z_p,
\]
are all equal. Consequently, the following are equivalent:
\begin{enumerate}
    \item $\Phi_d(X)\mid m_A(X)$;
    \item $\Phi_d(X)\mid m_{A_j}(X)$ for every $j\in\mathbb Z_p$;
    \item every $p$-level $A_j$ satisfies the $d$-cube rule.
\end{enumerate}
\end{lem}

\begin{proof}
By \eqref{eq:cyclotomic-projection}, it is enough to work in
$\mathbb Z_{dp}\cong\mathbb Z_d\times\mathbb Z_p$.

Let $\Delta$ denote a $d$-cuboid, and let $\Delta_j$ denote its translate in the
$j$th $p$-level. If $j\neq j'$, then
\[
    \Delta_j-\Delta_{j'}
\]
is a $dp$-cuboid. 
Since
$\Phi_{dp}\mid m_A$, Proposition~\ref{prop:cuboid-criterion} gives
\[
    A^{dp}[\Delta_j-\Delta_{j'}]=0.
\]
Therefore
\[
    A_j^d[\Delta]=A_{j'}^d[\Delta]
\]
for all $j,j'\in\mathbb Z_p$. Denote this common integer by
$c_\Delta$.

The evaluation of the projection of $A$ to $\mathbb Z_d$ is the sum of
the evaluations over all $p$-levels. Hence
\begin{equation}
\label{eq:level-evaluation}
    A^d[\Delta]
      =
    \sum_{j\in\mathbb Z_p}A_j^d[\Delta]
      =
    p\,c_\Delta.
\end{equation}
By Proposition~\ref{prop:cuboid-criterion},
\[
    \Phi_d\mid m_A
\]
holds if and only if $A^d[\Delta]=0$ for every $d$-cuboid $\Delta$.
By \eqref{eq:level-evaluation}, this is equivalent to
$c_\Delta=0$ for every $\Delta$, which in turn is equivalent to the
$d$-cube rule on every level. A final application of
Proposition~\ref{prop:cuboid-criterion} proves the equivalence with
$\Phi_d\mid m_{A_j}$ for every $j$.
\end{proof}

\begin{cmt}
\label{rem:levelwise-attribution}
No novelty is claimed for Lemma~\ref{lem:levelwise-cube-rule}. In the
square-free setting, the implication
\[
    \Phi_{dp}\mid m_A,\quad \Phi_d\mid m_A
    \quad\Longrightarrow\quad
    \Phi_d\mid m_{A_j}\quad\text{for every }j
\]
appears in \cite[Lemma~2.13]{KissMalikiosisSomlaiVizer2022}. The proof
above is the direct cuboid argument underlying that result; compare
also \cite[Proposition~4.2 and Lemma~6.1]{LabaMarshall2022}.
\end{cmt}

Taking $d=k$ in Lemma~\ref{lem:levelwise-cube-rule} gives the form that
will be used most often:
\[
    \Phi_{kp}(X)\mid m_A(X)
    \quad\Longrightarrow\quad
    \left[
        \Phi_k(X)\mid m_A(X)
        \ \Longleftrightarrow\
        \text{every $p$-level satisfies the $k$-cube rule}
    \right],
\]
where $k$ is assumed square-free when the terminology ``cube rule'' is
used.

The same argument gives a useful cardinality bound.

\begin{cor}[\L aba--Marshall bound]
\label{cor:laba-marshall}
Let $A$ be a nonnegative multiset in $\mathbb Z_M$. Suppose that
$p\nmid d$, $dp\mid M$, and
\[
    \Phi_{dp}(X)\mid m_A(X),
    \qquad
    \Phi_d(X)\nmid m_A(X).
\]
Then
\[
    |A|\geq p.
\]
\end{cor}

\begin{proof}
Since $\Phi_d\nmid m_A$, the cube rule provides a $d$-cuboid
$\Delta$ such that
\[
    A^d[\Delta]\neq0.
\]
By \eqref{eq:level-evaluation},
\[
    p\mid A^d[\Delta].
\]
As $A^d[\Delta]$ is a nonzero integer,
\[
    \left| A^d[\Delta] \right|\geq p.
\]
On the other hand, since $A$ has nonnegative multiplicities and the
coefficients of $\Delta$ belong to $\{-1,0,1\}$,
\[
p \leq    \bigl|A^d[\Delta]\bigr|
       \leq |A|.
\]
\end{proof}
Corollary~\ref{cor:laba-marshall} is the one-prime case of
\cite[Corollary~6.4]{LabaMarshall2022}. A related multi-prime form of Corollary~\ref{cor:laba-marshall} states
that if $p_1,\ldots,p_s$ are distinct primes not dividing $d$,
\[
    \Phi_{dp_i}(X)\mid m_A(X)
    \quad (i=1,\ldots,s),
    \qquad
    \Phi_d(X)\nmid m_A(X),
\]
then
\[
    |A|\geq p_1\cdots p_s;
\]
see \cite[Corollary~6.5]{LabaMarshall2022}.

\section{Inductive step}
\label{sec:large-prime}

Let $n$ be a positive integer and let $p>n$ be prime. Since
$\gcd(n,p)=1$, we will use both identifications
\[
    \mathbb Z_{np}
       \cong
    \mathbb Z_n\times\mathbb Z_p.
\]
The cyclic-group representation $\Z_{np}$ will be used for cyclotomic divisibility of mask the polynomial of a subset $A$ of $\Z_{np}$ and, while the product representation $\Z_n \times \Z_p$ will be used
to describe intersection of $A$ with cosets of $\Z_n$, which we also call \emph{levels} over $\mathbb Z_p$.

We first make explicit a projection argument that will be used when
$p$ does not divide the cardinality of the spectral set. A version of
this argument is implicit in \cite[Lemma~3.1]
{FallonKissMayeliSomlai2023} The main goal of that paper was also to build an inductive framework, which is finally developed in the present paper.
\begin{lem}
\label{lem:multiplication-preserves-spectrality}
Let
\[
    N=kn,
\]
where $k$ is prime and $k\nmid n$. Let $A\subseteq\mathbb Z_N$ be
spectral, and suppose that multiplication by $k$ is injective on $A$.
Suppose, in addition, that
\begin{equation}
\label{eq:cyclotomic-descent}
    \Phi_{dk}(X)\mid m_A(X)
    \quad\Longrightarrow\quad
    \Phi_d(X)\mid m_A(X)
\end{equation}
for every divisor $d$ of $n$, including $d=1$, where
$\Phi_1(X)=X-1$. Then $kA$ is spectral as a subset of the subgroup
\[
    k\mathbb Z_N\cong\mathbb Z_n.
\]

More precisely, if $\Lambda$ is a spectrum of $A$, then the restrictions
to $k\mathbb Z_N$ of the characters indexed by $\Lambda$ form a spectrum
of $kA$.
\end{lem}

\begin{proof}
Let $\Lambda$ be a spectrum of $A$, and let
$\lambda,\lambda'\in\Lambda$ be distinct. Let
  $  r=\operatorname{o}_N(\lambda-\lambda')$.
Since $(A,\Lambda)$ is a spectral pair,$
    \Phi_r(X)\mid m_A(X)$.
Clearly
\begin{equation}\label{eq:indukcio0}
    \begin{aligned}
 \sum_{b\in kA}
     e^{2\pi i(\lambda-\lambda')b/N}
 &=
 \sum_{a\in A}
     e^{2\pi i k(\lambda-\lambda')a/N}.
\end{aligned}
\end{equation}

If $k\nmid r$, then
\[
    \operatorname{o}_N\bigl(k(\lambda-\lambda')\bigr)=r,
\]
and therefore the last sum is zero because
$\Phi_r\mid m_A$.
If $k\mid r$, then, since $k\nmid n$, we may write
$    r=dk$
for some $d\mid n$. In this case,
\[
    \operatorname{o}_N\bigl(k(\lambda-\lambda')\bigr)=d.
\]
By \eqref{eq:cyclotomic-descent},
\[
    \Phi_d(X)\mid m_A(X),
\]
so \eqref{eq:indukcio0} is again zero.
Thus the restrictions of the characters indexed by $\Lambda$ are
pairwise orthogonal on $kA$. Since the multiplication by $k$  is injective on $A$, we have that $kA$ is a proper so we have
\[
    |\Lambda|=|A|=|kA|,
\]
and then $\Lambda$ is a spectrum for $kA$.
\end{proof}
For a set $
    B\subseteq \mathbb Z_n\times\mathbb Z_p
$
and $j\in\mathbb Z_p$, we write
\[
    B_j
      =
    \{x\in\mathbb Z_n:(x,j)\in B\}
\]
for the $j$th $p$-level of $B$.

The following proposition contains the main structural observation in
the case where $p$ divides the cardinality.

\begin{prp}[A common spectrum for all levels]
\label{prop:common-spectrum-levels}
Let $p>n$ be prime, and let $(A,\Lambda)$ be a spectral pair in
$
    \mathbb Z_n\times\mathbb Z_p
$.
Suppose that
\[
    |A|=|\Lambda|=p\ell>0.
\]
Then all the levels $A_j$ and $\Lambda_j$ have cardinality $\ell$.
Moreover, there exists a set
$
    E\subseteq\mathbb Z_n,
    \qquad |E|=\ell,
$
such that $E$ is a spectrum of $A_j$ for every
$j\in\mathbb Z_p$.
\end{prp}

\begin{proof}
Since
$
    |\Lambda|=p\ell\geq p>n
$,
the projection of $\Lambda$ onto $\mathbb Z_n$ is not injective.
Therefore, there are
\[
    (x,s),(x,t)\in\Lambda,
    \qquad s\neq t.
\]
The order of the difference of these two elements of $\Lambda$ is Consequently $
    \Phi_p(X)\mid m_A(X)$. Then Corollary \ref{cor:prime-divides-cardinality} implies the integers $|A_j|$ are all equal.
By symmetry of spectral pairs, $(\Lambda,A)$ is also a spectral pair, so the same argument applied to
$(\Lambda,A)$ gives
$
    |\Lambda_j|=\ell$
    for every $j\in\mathbb Z_p
$.

For $x\in\mathbb Z_n$, define
\[
    r(x)
       =
    \bigl|\{j\in\mathbb Z_p:x\in\Lambda_j\}\bigr|.
\]
Call a level $\Lambda_j$ \emph{exceptionally bad} if it contains an element
$x$ with $r(x)=1$. There are at most $n$ exceptionally bad levels. Indeed,
to every exceptional level one may assign an element of
$\mathbb Z_n$ that occurs in that level and in no other level, and
distinct exceptional levels must be assigned distinct elements.

Since $p>n$, there is at least one non-exceptionally bad level $\Lambda_{j_0}$. Set
\[
    E=\Lambda_{j_0}.
\]
Then
$    |E|=\ell$,
and every $e\in E$ belongs to at least one other level $\Lambda_s$,
with $s\neq j_0$.

We claim that $E$ is a spectrum of every $A_j$. Let $e,f\in E$ be
distinct, and put
$
    d=\operatorname{o}_n(e-f)
$.
Because
\[
    (e,j_0),(f,j_0)\in\Lambda,
\]
their difference has order $d$ in
$\mathbb Z_n\times\mathbb Z_p$. Hence spectrality gives
\begin{equation}
\label{eq:phi-d-from-same-level}
    \Phi_d(X)\mid m_A(X).
\end{equation}
Since $\Lambda_{j_0}$ is non-exceptional, there is an
$s\neq j_0$ such that
$(e,s)\in\Lambda$.
The difference
$ (e,s)-(f,j_0) \in \Lambda-\Lambda$
has order
\[
    \operatorname{lcm}(d,p)=dp,
\]
because $d\mid n$ and $p\nmid n$. Therefore,
\begin{equation}
\label{eq:phi-dp-from-different-levels}
    \Phi_{dp}(X)\mid m_A(X).
\end{equation}
By Lemma~\ref{lem:levelwise-cube-rule}, the divisibilities
\eqref{eq:phi-d-from-same-level} and
\eqref{eq:phi-dp-from-different-levels} imply
\[
    \Phi_d(X)\mid m_{A_j}(X)
    \qquad
    \text{for every }j\in\mathbb Z_p.
\]
Thus the characters indexed by $E$ are pairwise orthogonal on every
$A_j$. Since
\[
    |E|=|A_j|=\ell,
\]
the set $E$ is a spectrum of $A_j$ for every $j$.
\end{proof}
\begin{rmk}
In fact, we should have defined the exceptionally bad level such that it contains two points for which \(r(x)=1\), which immediately shows that half of the lemma can certainly be further improved, but it is most understandable in this simplest form.
\end{rmk}

We can now prove the main result of the paper. Let us recall Theorem \ref{thm:induction}
\begin{thm*}
Let $n$ be square-free and let $p>n$ be prime. Suppose that
\[
    \mathrm{S\!-\!T}(\mathbb Z_n)
\]
holds. Then
\[
    \mathrm{S\!-\!T}
       (\mathbb Z_n\times\mathbb Z_p)
\]
holds. 
\end{thm*}

\begin{proof}
Let
\[
    G=\mathbb Z_n\times\mathbb Z_p
       \cong \mathbb Z_{np},
\]
and let $(A,\Lambda)$ be a spectral pair in $G$. We distinguish two
cases.

\medskip
\noindent
\textbf{Case 1: $p\nmid |A|$.}
We first show that multiplication by $p$ is injective on $A$. Suppose
otherwise that
    $pa=pa'$
for distinct $a,a'\in A$. Then $a-a'$ is a nonzero element of order $p$.
Since $(\Lambda,A)$ is also a spectral pair, it follows that
$   \Phi_p(X)\mid m_\Lambda(X)$. By Corollary \ref{cor:prime-divides-cardinality} $p \mid \vert \Lambda \vert =\vert A \vert$, contrary to the assumption  $p \nmid \vert A\vert$. 
Thus $pA$ is a set contained in the subgroup
$
    p\mathbb Z_{np},
$
which has cardinality $n$. Consequently,
\[
    |A|=|pA|\leq n<p.
\]
Let $d\mid n$, and suppose that
\[
    \Phi_{dp}(X)\mid m_A(X).
\]
If $\Phi_d\nmid m_A$, then Corollary~\ref{cor:laba-marshall} gives
\[
    |A|\geq p,
\]
contradicting $|A|<p$. Hence
\[
    \Phi_{dp}(X)\mid m_A(X)
    \quad\Longrightarrow\quad
    \Phi_d(X)\mid m_A(X)
\]
for every $d\mid n$.
Lemma~\ref{lem:multiplication-preserves-spectrality} now
shows that $pA$ is spectral in
$
    p\mathbb Z_{np}\cong\mathbb Z_n$.

Let
\[
    B=\{x\in\mathbb Z_n:(x,j)\in A
         \text{ for some }j\in\mathbb Z_p\}
\]
be the projection of $A$ onto $\mathbb Z_n$. The injectivity proved
above shows that there is a function
\[
    f:B\longrightarrow\mathbb Z_p
\]
such that
\[
    A=\{(x,f(x)):x\in B\}.
\]
Since multiplication by $p$ is an automorphism of $\mathbb Z_n$, the
spectrality of $pA$ implies that $B$ is spectral in $\mathbb Z_n$\footnote{This observation treats the annoying fact that the projection to the first coordinate and multiplication by $p$ is not exactly the same operation}.
By $\mathrm{S\!-\!T}(\mathbb Z_n)$, there is a set
$C\subseteq\mathbb Z_n$ such that
$
    B\oplus C=\mathbb Z_n$.
It follows immediately that
$
    A\oplus(C\times\mathbb Z_p)
       =
    \mathbb Z_n\times\mathbb Z_p$.

\medskip
\noindent
\textbf{Case 2: $p\mid |A|$.}
Write
$
    |A|=p\ell.$
By Proposition~\ref{prop:common-spectrum-levels}, there is a set
$    E\subseteq\mathbb Z_n,
    \qquad |E|=\ell,
$
which is a common spectrum of every level $A_j$.
The assumption $\mathrm{S\!-\!T}(\mathbb Z_n)$ therefore implies that
each $A_j$ tiles $\mathbb Z_n$. In particular, $\ell\mid n$.

Since $n$ is square-free, the square-free case of the
Coven--Meyerowitz conjecture (or Tijdeman's theorem) gives the same standard subgroup of order $\Z_{\frac{n}{l}}$ as a tiling complement
for every tile of cardinality $\ell$. More precisely,
\[
    A_j\oplus \ell\mathbb Z_n
       =
    \mathbb Z_n
    \qquad
    \text{for every }j\in\mathbb Z_p.
\]
It follows level by level that
\[
    A\oplus
       \bigl(\ell\mathbb Z_n\times\{0\}\bigr)
       =
    \mathbb Z_n\times\mathbb Z_p.
\]
Thus $A$ tiles $G$ in both cases.
\end{proof}

\begin{cmt}
\label{rem:where-squarefree-used}
The square-free hypothesis is used only in the final step of Case~2,
where it provides a common tiling complement for all the levels
$A_j$. Proposition~\ref{prop:common-spectrum-levels} itself shows that
the levels have a common spectrum without using the
Coven--Meyerowitz classification. Thus, for a more general modulus
$n$, the same argument works whenever every family of spectra of one
set in $\mathbb Z_n$ admits a common tiling complement.
\end{cmt}

\section{Extension of this work}
\textbf{Other groups satisfying the Coven-Meyerowitz conjecture:}

The square-free assumption in
Theorem~\ref{thm:induction} is used only in the case
$p\mid |A|$, in order to obtain a common tiling complement for the
$p$-levels of $A$. This suggests that the theorem should extend more
generally to other finite (cyclic) groups for which the Coven--Meyerowitz conjecture is
known. Indeed, Proposition~\ref{prop:common-spectrum-levels} produces a
common spectrum for all the levels, while the Coven--Meyerowitz
conditions associate to each level a standard tiling complement
determined by its prime-power cyclotomic divisors. To complete such an
extension, one would need an additional compatibility argument showing
that the common spectrum forces 
the standard complements to agree across the levels. The results of
\L aba and Londner therefore provide natural non-square-free families in
which to pursue this extension.

\textbf{Answering a question of Zhang (and Fan):}

Fan and Zhang \cite{FanZhangPeriodicity} recently proved that
\[
    \mathbb Z_3^4
\]
is a T--S group but not an S--T group.  They also observe that
no example is currently known for which the opposite one-sided
phenomenon occurs, namely an S--T group that is not a T--S group.
In particular, this example of Fan and Zhang is the first
known finite abelian group for which Fuglede's conjecture holds in
exactly one direction \cite{FanZhangPeriodicity}.

The failure of the S--T property is automatically preserved under
direct products (and probably other group extensions as well). More precisely, if $E\subseteq G$ is spectral but does
not tile $G$, and $H$ is any finite abelian group, then
\[
    E\times H
\]
is spectral but does not tile $G\times H$. Indeed, if $\Lambda$ is a
spectrum of $E$, then $\Lambda\times\widehat H$ is a spectrum of
$E\times H$. On the other hand, a tiling of $G\times H$ by
$E\times H$ would induce a tiling of $G$ by $E$, which is impossible.

Consequently, in order to construct infinitely many groups for which
T--S holds but S--T fails, it remains only to preserve the T--S
property. A natural objective is therefore to prove a large-prime
extension of the form
\[
    \mathrm{T\!-\!S}(G)
    \quad\Longrightarrow\quad
    \mathrm{T\!-\!S}(G\times\mathbb Z_q)
\]
for $G=\mathbb Z_3^4$ and all sufficiently large primes $q$. Iterating
such a result with pairwise distinct primes would produce an infinite
family
\[
    \mathbb Z_3^4\times
    \mathbb Z_{q_1q_2\cdots q_r},
    \qquad r\geq1,
\]
of T--S groups which are not S--T groups.

While I believe this is doable, the rapid pace of developments surrounding Fuglede's conjecture convinced me to first present the inductive result on Fuglede's conjecture which is likely to be of broader interest.


\section{Acknowledgement}
The basic idea behind this article is the author’s own, but of course, much of this appears in the author’s earlier work, co-authored with Fallon, Kiss, and Mayeli, so I am very grateful to these people for our previous discussions. I also discussed the method with Máté Matolcsi and Itay Londner, and I am very grateful to them as well for those conversations. In addition, writing it in this form and at this pace would not have been possible without OpenAI’s ChatGPT and Google’s Gemini, which provided great assistance in verifying the proof, collecting precisely the relevant literature and earlier results according to the author's taste, and formatting the text—though I am, of course, responsible for any errors in the final version.

\bibliographystyle{plain}
\bibliography{references}
\end{document}